\documentclass[10pt]{amsart}
\usepackage[centertags]{amsmath}
\usepackage{amsfonts}
\usepackage{amssymb}
\usepackage{amsthm}
\usepackage{newlfont}
\usepackage[all]{xy}

\newcommand{\QQ}{\mathbb{Q}}
\newcommand{\ZZ}{\mathbb{Z}}
\newcommand{\CC}{\mathbb{C}}
\newcommand{\RR}{\mathbb{R}}
\newcommand{\FF}{\mathbb{F}}
\newcommand{\PP}{\mathbb{P}}
\renewcommand{\AA}{\mathbb{A}}

\newcommand{\Gal}{\operatorname{Gal}}

\newcommand{\PGL}{\operatorname{PGL}}

\newcommand{\Rat}{\operatorname{Rat}}

\renewcommand{\epsilon}{\varepsilon}
\renewcommand{\phi}{\varphi}

\newcommand{\HH}{\mathbb{H}}
\newcommand{\Berk}{\textup{Berk}}
\newcommand{\PBerk}{\PP^1_{\Berk}}
\newcommand{\ABerk}{\AA^1_{\Berk}}
\newcommand{\HBerk}{\HH_{\Berk}}
\newcommand{\Dbar}{\overline{D}}
\newcommand{\PK}{\PP^1(K)}

\DeclareMathOperator{\rad}{rad}
\DeclareMathOperator{\charact}{char}

\newtheorem{theorem}{Theorem}[section]

\newtheorem{lemma}[theorem]{Lemma}
\newtheorem{cor}[theorem]{Corollary}

\theoremstyle{remark}
\newtheorem{remark}[theorem]{Remark}
\newtheorem{example}[theorem]{Example}

\theoremstyle{definition}
\newtheorem{defn}[theorem]{Definition}

\begin{document}

\newcounter{bean}

\title[Attracting cycles in $p$-adic dynamics]{Attracting cycles in $p$-adic dynamics and height bounds for post-critically finite maps}
\author[R.~Benedetto, P.~Ingram, R.~Jones, and A.~Levy]{Robert Benedetto, Patrick Ingram, Rafe Jones, and Alon Levy}
\date{\today}

\address{Department of Mathematics, Amherst College, Amherst, MA}
\address{Department of Mathematics, Colorado State University, Fort Collins, CO}
\address{Department of Mathematics, Carleton College, Northfield, MN}
\address{Department of Mathematics, University of British Columbia, Vancouver, Canada}

\begin{abstract}
A rational function of degree at least two with coefficients in an algebraically closed field is post-critically finite (PCF) if and only if all of its critical points have finite forward orbit under iteration.  We show that the collection of PCF rational functions is a set of bounded height in the moduli space of rational functions over the complex numbers, once the well-understood family known as flexible
Latt\`es maps is excluded. As a consequence, there are only finitely many conjugacy classes of non-Latt\`es PCF rational maps of a given degree defined over any
given number field. The key ingredient of the proof is a non-archimedean version of Fatou's classical result that every attracting cycle of a rational function over
$\CC$ attracts a critical point.
\end{abstract}

\maketitle

\section{Introduction}\label{intro}

A rational function $\phi \in \CC(z)$ is \textit{post-critically finite} (PCF) if all of its critical points in $\PP^1(\CC)$ have finite forward orbits under iteration
of $\phi$.  Over $\CC$, where the orbits of the critical points are known to play a central role in the global dynamics of a map, it is not surprising that PCF maps
exhibit interesting behavior. Since Thurston's foundational result on PCF maps \cite{thurston}, a growing body of work has focused on their properties \cite{BFH, pilgrim, koch2, cfpp, hp, koch1, pilgrim2, poirier, rong, selinger}.  In this article, we study PCF maps from an arithmetic point of view, where their novel properties are only beginning to be explored.

Our interest here lies in the distribution of PCF maps in the moduli
space $\mathcal{M}_d$ of rational functions of degree $d\geq 2$ up to
change of variables; see Section~\ref{elaboration} for more precise
definitions. Thurston's result implies that apart from a
well-understood class of PCF maps associated to elliptic curves, known
as the flexible Latt\`es maps, there are only finitely many conjugacy
classes of rational maps defined over $\CC$ whose critical points each
have orbits of length not exceeding $N$, for any given integer $N$.
Moreover, all such maps are
$\overline{\QQ}$-rational points in the moduli space $\mathcal{M}_d$.
Hence, the PCF points in $\mathcal{M}_d$
consist of the Latt\`es locus plus a countable set, and they
are therefore in a precise sense a sparse subset of
$\mathcal{M}_d(\CC)$. However, this fact does not
\emph{a priori} preclude the possibility
that many, or even most, points in $\mathcal{M}_d(\overline{\QQ})$
give conjugacy classes of PCF maps.  Our main result shows this is not
the case, by bounding the height
of PCF points, and thereby proving a conjecture of Silverman~\cite[Conjecture~6.30, p.~101]{barbados}.  If $U\subseteq \mathcal{M}_d$ denotes the subvariety of non-Latt\`{e}s maps, then work of McMullen~\cite{mcmullen} shows that the map $\Lambda:U\to\PP^M$, taking a rational function to the multiplier spectrum of its $N$-periodic points, is finite when $N$ is large enough. The height referred to in the next theorem is the pull-back of the Weil height on $\PP^M$ by $\Lambda$ for some sufficiently large $N$.

\begin{theorem} \label{th:pcfmain}
For each $d \geq 2$, the PCF locus in $\mathcal{M}_d(\overline{\QQ})$
consists of the flexible Latt\`{e}s locus, plus a set of bounded
height. In particular, for any fixed integer $B\geq 1$,
there are, up to change of variables over $\overline{\QQ}$,
only finitely many PCF rational functions of degree $d$
that are not flexible Latt\`es maps but
that may be defined over a number field of degree at most $B$.
\end{theorem}

From a number-theoretic perspective, one of the motivations for
studying PCF points in $\mathcal{M}_d$ is an analogy with
CM points in the moduli space of elliptic curves.
If $\varphi$ is defined over a number field $L$, one
may associate to $\phi$ an {\em arboreal Galois representation} via
the natural action of $\Gal(\overline{L}/L)$ on the infinite tree of
all preimages, under the iterates of $\varphi$, of a fixed
$L$-rational point.
The image of this representation is much smaller for PCF maps than
for typical rational functions \cite{hajir, BJ2}.
By comparison,
the $\ell$-adic Galois representations
of elliptic curves with complex multiplication
have much smaller images than
those attached to elliptic curves without CM.

The second claim in Theorem~\ref{th:pcfmain} is very similar to a statement for $j$-invariants of CM elliptic curves, which follows from class field theory and the Gauss Conjecture, originally
proven by Heilbronn (we note, however, that the set of $j$-invariants of CM elliptic curves is not of bounded height). In the case of polynomials, the second author \cite{pcfpoly} has already established Theorem \ref{th:pcfmain}, which follows from stronger
results relating the height of the coefficients of a polynomial to the rate of growth of the heights in its critical orbits, but the arguments in \cite{pcfpoly} seem unlikely to generalize to rational functions.  The proof of Theorem~\ref{th:pcfmain} follows a fundamentally different approach, which comes from studying the height of the multipliers of periodic cycles.

Let $K$ be a field, let $\phi\in K(z)$ be a rational function,
and let $\gamma\in\PK$.
If $\varphi(\gamma) = \gamma$, we say $\gamma$ is a \emph{fixed point}
of $\phi$.
In that case, by a change of coordinates, we may assume
that $\gamma\neq\infty$, and define the \emph{multiplier} of $\gamma$
to be $\lambda=\phi'(\gamma)\in K$.
More generally, if $\phi^n(\gamma)=\gamma$
for some $n\geq 1$, where $\phi^n$ denotes the $n$-fold
composition $\phi\circ\cdots\circ\phi$, then we say $\gamma$
is a \emph{periodic point} of $\phi$ of period $n$; and if $n$
is the minimal period of $\gamma$, then we define
the \emph{multiplier} of $\gamma$ to be $\lambda=(\phi^n)'(\gamma)\in K$.
The multiplier is invariant under coordinate change,
and it is the same for each point $\phi^k(\gamma)$ in the
forward orbit, or periodic cycle, of the periodic point $\gamma$.
If $K$ is equipped with an absolute value $|\cdot|$, we say that
the orbit of the periodic point $\gamma$ is
\emph{attracting} if $0 \leq |\lambda| < 1$; the case $\lambda=0$ is referred to as the \emph{superattracting} case.
We say that $\gamma$
\emph{attracts} $x \in K$ if
$\lim_{m\to\infty}\varphi^{nm}(x) = \phi^k(\gamma)$
for some $0\leq k\leq n-1$, and that
$\gamma$ \textit{strictly attracts} $x\in K$
if $\gamma$ attracts $x$, but
$\varphi^l(x) \neq \gamma$ for all $l \geq 1$.

The main engine in our proof of Theorem~\ref{th:pcfmain}
is the following result.

\begin{theorem}\label{th:main}Let $K$ be an algebraically closed field which is complete with respect to a non-trivial non-archimedean absolute value $|\cdot|$. Let $p\geq 0$ be the residue characteristic of $K$, let $d\geq 2$ be an integer, and assume either that $\charact K=0$ or that $\charact K > d$.  Define the real number $\epsilon=\epsilon_{p,d}\leq 1$ to be $$\epsilon = \min\{ |m|^d : 1\leq m \leq d\},$$ which is clearly positive. Let $\varphi(z) \in K(z)$ be a rational function of degree $d$, and let $\gamma$ be a fixed point of $\varphi$ satisfying $$0 < |\varphi'(\gamma)| < \epsilon.$$ Then there is a critical point of $\varphi$ which is strictly attracted to $\gamma$.\end{theorem}

Note that if $p=0$ or $p > d$, then the definition above gives $\epsilon_{p,d} = 1$. In addition, we will see in Theorem~\ref{th:polys} that if $\varphi(z)\in K[z]$ is a polynomial, then the constant $\epsilon$ in Theorem~\ref{th:main} can be improved to $\epsilon_{p, d}^{\mathrm{poly}}=\min\{|m|^m:1\leq m\leq d\}$, which is sharp in that case. For general rational functions, however, we will see in Section~\ref{sec:sharp} that the bound $\epsilon_{p,d}$ is not sharp.

Note also that Theorem~\ref{th:main} applies to function fields of characteristic $0$ or $p > d$, with $\epsilon = 1$, but says nothing about maps defined over function fields of characteristic $p$ with $0 < p \leq d$, since then the value of $\epsilon$ is $0$. Indeed, we can find PCF maps over such fields with arbitrarily small nonzero multipliers: the map $z^{p} + t^{n}z$, defined over the completed algebraic closure of $\FF_p((t))$, is PCF and has a fixed point at $z=0$ of multiplier $t^n$, which can be made arbitrarily small by increasing $n$.

It is a classical result from complex dynamics that over $\CC$, every
attracting cycle is either super-attracting, or strictly attracts a critical point
\cite[Theorem~9.3.1]{beardon}.
The proof uses complex analysis in a fundamental
way. The analogous statement over a non-archimedean field of positive
residue characteristic, however, is false; for instance, the map
$\varphi(z) = z^{p}$, defined over the completion $\CC_{p}$ of the
algebraic closure of $\QQ_{p}$, has the curious property that every
cycle is attracting, with respect to the natural extension of the
$p$-adic absolute value.  Thus, although the full strength of the
classical complex result does not carry over to non-archimedean
fields, Theorem~\ref{th:main} says that in the non-archimedean case, a
non-superattracting
cycle strictly attracts a critical point provided that the cycle is
\textit{sufficiently} attracting.

We will deduce Theorem \ref{th:pcfmain} from Theorem \ref{th:main}
in Section~\ref{elaboration}, but we briefly sketch the idea here.
By definition, a PCF map defined over a number field $L$ cannot have a
critical point strictly attracted to a periodic point over \emph{any}
completion $L_v$, as $v$ varies over the places of $L$.  Thus if
$\varphi$ is a PCF map defined over $L$, then Theorem~\ref{th:main}
provides an upper bound on the $v$-adic absolute value of the
reciprocal of the multiplier of any cycle of length $n$, for such a
cycle consists of fixed points of $\varphi^n$.  The local bounds
induce a bound on the height of any such multiplier. For example, in the case $d = 2$, we obtain:
\begin{cor}\label{cor:quadratic} Let $\varphi(z) \in \CC(z)$ be a rational function of degree $2$, let $h$ denote the logarithmic Weil height, and suppose that $\varphi$ is PCF. If $\lambda$ is the
multiplier of a fixed point of $\varphi$, then $h(\lambda) \leq \log 4$.
\end{cor}
By a basic property of the Weil height
(see Definition~\ref{def:height}), it follows
that the multiplier of a cycle of $\phi$ of length $n$ can take only finitely many values in $L$.
A celebrated theorem of McMullen
\cite{mcmullen} states that for $N$ chosen large enough relative to $d$,
the multipliers of all $n$-cycles with $1 \leq n \leq N$ essentially
give a parameterization of $\mathcal{M}_{d}(\CC)$ off of the locus
of flexible Latt\`es maps.
Thus, the finite list of
possible multipliers of cycles of PCF maps is
realized by only finitely many non-Latt\`{e}s points in
$\mathcal{M}_d$, proving Theorem \ref{th:pcfmain}.

In light of Theorem~\ref{th:pcfmain}, it is natural to wonder whether the finite set alluded to is effectively computable; the results in \cite{pcfpoly}, for example, provide an effective algorithm for computing the corresponding set in the polynomial case, and work of Goldfeld \cite{goldfeld1, goldfeld2} and Gross-Zagier \cite{gz} on the Gauss Conjecture provides an effective way of enumerating $j$-invariants of CM elliptic curves. Since the elements of a set of bounded height in affine space may be effectively enumerated, the question reduces to whether or not McMullen's theorem may be made effective, that is, whether there is an algorithm to compute the list of rational functions with a given multiplier spectrum.  In general, an effective result is not known, but when $\phi$ has degree two, work of Milnor \cite{milnor} and Silverman \cite{silverman} shows that $\phi$ is determined up to conjugacy by the multipliers of its fixed points. 
Corollary~\ref{cor:quadratic} can thus be used to explicitly compute, for
any $B\geq 1$, the finitely many conjugacy classes of quadratic
rational maps having a representative whose coefficients lie in an
extension of $\QQ$ of degree at most $B$. Recently Manes and Yap \cite{manes} have done just this in the case $B = 1$, finding twelve conjugacy classes of rational maps with a representative defined over $\QQ$.

If $\mathrm{Per}_n(\lambda)\subseteq \mathcal{M}_2$ is the curve consisting of rational maps of degree two, with a marked $n$-cycle of multiplier $\lambda$, then DeMarco has conjectured that $\mathrm{Per}_n(\lambda)$ contains only finitely many PCF maps, unless $\lambda=0$ (see \cite[p.~112]{barbados}).  Indeed, DeMarco points out that $\mathrm{Per}_n(\lambda)$ contains no PCF maps at all if $0<|\lambda|<1$, in the complex absolute value, by the aforementioned theorem of Fatou.  The calculations which prove Corollary~\ref{cor:quadratic} show that the situation is analogous for non-archimedean places.  Specifically, if $\lambda\in\CC^*$ is algebraic and $\mathrm{Per}_n(\lambda)$ contains a PCF map, then $|\lambda|_v\geq 1$ for every archimedean place $v$, as well as every $p$-adic place with $p\neq 2$, while $|\lambda|_v\geq \frac{1}{4}$ for every 2-adic place.  Note that there are no transcendental values  $\lambda\in\CC^*$ such that $\mathrm{Per}_n(\lambda)$ contains a PCF map, by Thurston's rigidity results.

The proof of Theorem \ref{th:main} uses the Berkovich projective line, and is developed in Sections \ref{sec:nonarch} and \ref{proof}.  Another outcome of the method is the following
fact about non-archimedean analysis, which may be of independent interest.
\begin{theorem}
\label{th:nondyn}
Let $K$ be a field as in Theorem~\ref{th:main},
let $\varphi(z)\in K(z)$ be a rational function of degree $d\geq 2$
for which $\phi(\infty)=\infty$,
and let $a\in K$.  Suppose that $a$ is not a pole of $\varphi$,
and let $r$
be the radius of the largest punctured disk about $z=a$ which is
disjoint from $\varphi^{-1}(\{\varphi(a),\infty\})$.
Then there is a critical point $\beta$ of $\varphi$ satisfying
$$|\varphi(\beta)-\varphi(a)|\leq \epsilon^{-1} r|\varphi'(a)|,$$
where $\epsilon$ is the same constant from Theorem~\ref{th:main}.
\end{theorem}

These results should be compared with recent results of Faber~\cite{faber} on a related topic.  He shows that if a rational function over a non-archimedean field has more than one zero in a closed disk $D$ of given radius, then there is a critical point in a disk not much larger (in fact in $D$ in the case $p = 0$ or $p > d$). We show the same, except that on the one hand our bound for ``not much larger'' in the case $0 < p \leq d$ is much less sharp, and on the other hand the critical point's corresponding critical value is in the interior of a disk that is also not too large compared with $D$ (in fact, a strictly smaller disk if $p = 0$ or $p > d$).

Theorem~\ref{th:main} is stated for fixed points, but we may apply it
to $\varphi^{n}$ and thereby extend the result to $n$-cycles.
It follows that an $n$-cycle strictly attracts a critical point when it has
multiplier $\lambda$ satisfying $0<|\lambda|<\epsilon_{p, d^{n}}$.
This bound is sufficient for the purposes of
Theorem~\ref{th:pcfmain}, for which we
fix $n$; but if we wish to consider cycles of arbitrary length, we
are hampered by the fact that $\epsilon_{p, d^{n}}$ becomes
arbitrarily small.  It turns out, however, that $n$th iterates of
rational functions of degree $d$ are not typical amongst rational
functions of degree $d^n$, and that we can significantly improve the
above estimates for $n$-cycles when $p>d$.

\begin{theorem}\label{introper}
Let $p$ and $K$ be as in Theorem \ref{th:main}, and suppose
either that $p=0$ or that
$\varphi(z)\in K(z)$ has degree $d<p$.  If $\gamma$ is
an $n$-periodic point of $\phi$ with
$$0 < |(\varphi^n)'(\gamma)| < 1,$$
then there is a critical point of $\phi$ that is strictly
attracted to the cycle containing $\gamma$.
\end{theorem}

It is a classical result in holomorphic dynamics over $\CC$,
proven by both Fatou \cite{fatou} and Julia \cite{julia},
that any rational function $\varphi(z)\in\CC(z)$ of degree $d\geq 2$
has at most $2d-2$ attracting cycles; the key fact is that each
such cycle attracts a critical point, albeit possibly not strictly.
(Shishikura \cite{shishikura} later extended the same bound
of $2d-2$ to all non-repelling cycles.)
Theorem \ref{introper} thus allows us to give a
non-archimedean analogue of the classical result.

\begin{cor}\label{cor:attracting2d-2}Let $K$ be a non-archimedean field with residue characteristic $p$, and suppose that $\varphi(z)\in K(z)$ has degree $d \geq 2$ with $p > d$ or $p = 0$.  Then $\varphi$ has at most $2d - 2$ attracting cycles.\end{cor}

Rivera-Letelier has proven a similar result \cite[Corollaire~4.7 and Corollaire~4.9]{jrl} without the hypotheses on the characteristic, with a bound of $3d-3$. There is no hope of removing the hypothesis $p>d$ or $p=0$ from Corollary~\ref{cor:attracting2d-2}, as the map $z^p$ has infinitely many attracting cycles over any field of residue characteristic $p$. Even under those hypotheses, there is also no hope of bounding the number non-repelling cycles, as Shishikura did in the complex case, since any map with good reduction over a non-archimedean field has no repelling cycles.  Still, since $\CC_p$ is isomorphic (as a field) to $\CC$, and because \emph{rationally} indifferent cycles (i.e., cycles whose multiplier is a root of unity) remain rationally indifferent under base change, it follows from Corollary~\ref{cor:attracting2d-2} and Shishikura's result that under the same hypotheses, the number of attracting cycles plus the number of rationally indifferent cycles is at most $4d-4$.

With current technology, we know of no way to improve the bound of $4d-4$, even if we use the more generous counting scheme introduced by Epstein \cite{epstein}. The complex bound of $2d-2$ is false in the $p$-adic case even when so restricted to attracting and rationally indifferent cycles: one of our anonymous referees pointed out that $\phi(z) = z^2 - 3/4 \in \CC_3(z)$ has attracting fixed points at $\infty$ and $3/2$ and a rationally indifferent one at $-1/2$.

Another application of Theorem~\ref{introper} is to the case of a global function field.
\begin{cor}\label{cor:thurstonmodp} Let $K$ be a function field of characteristic $p$, and suppose that either $p=0$ or $p > d$.  If $\varphi(z) \in K(z)$ is a PCF
map of degree $d$, then the multipliers of all periodic points of $\varphi$ lie in the algebraic closure of the prime subfield of $K$.
\end{cor}

In particular, Corollary \ref{cor:thurstonmodp} implies that if McMullen's Theorem holds over the algebraic closure of the prime subfield of $K$, then every PCF map in $K$ is Latt\`{e}s or isotrivial. Over $\CC$, McMullen's proof uses Thurston rigidity and therefore we do not obtain anything new. However, in some cases a McMullen-type result is known without the use of Thurston rigidity, namely for quadratic rational maps over a global function field of odd characteristic. In such cases, we deduce that any PCF map that is not a flexible Latt\`es map is in fact defined over an algebraic extension of the field of constants, after a change of coordinates (see Corollary \ref{rigid2}).

Finally, we can reverse the question of Theorem~\ref{th:main} and ask what happens around repelling periodic points. It turns out that there exist PCF maps with repelling
points, even when the residue characteristic is $0$ or $> d$, where such maps could not have attracting points that are not superattracting. In personal communication, Rivera-Letelier
points to the example
\begin{equation}\label{rlex}
\phi(z) = -45\frac{3z + 5}{z^{2}(z - 9)} \in\CC_5(z)
\end{equation}
which is PCF but has two repelling fixed points --- one with absolute value $1$, and one with absolute value $|5|<1$.

However, if we consider polynomials only, it turns out that PCF maps do not have repelling points when the residue characteristic is $0$ or $> d$, and in fact they necessarily have potentially good
reduction. We use Newton polygons to prove this, in Theorem \ref{integrality}. The same method can be used to reprove and slightly generalize a result of
Epstein~\cite{epstein2}, which states that in residue characteristic $p$, PCF polynomials of $p$-power degree have potentially good reduction.

The structure of this article is as follows. In Section \ref{elaboration}, we give background results and deduce Theorem \ref{th:pcfmain} and Corollary \ref{cor:quadratic} from Theorem \ref{th:main}. In Section \ref{sec:nonarch}, we recall some useful facts about Berkovich space and non-archimedean analysis, and develop several lemmas that are instrumental in the proof of Theorem \ref{th:main}. In Section \ref{proof}, we complete the proofs of Theorems \ref{th:main} and \ref{th:nondyn}. Section \ref{sec:sharp} contains a discussion of the sharpness of the bound $\epsilon_{p,d}$ in Theorem \ref{th:main}. In Section \ref{perpts}, we prove Theorem \ref{introper} and deduce Corollary \ref{cor:thurstonmodp} from it. In Section~\ref{goodreduction} we address repelling periodic points of PCF maps. 


\section{Background on multipliers and the moduli space of dynamical systems} \label{elaboration}

We begin with some remarks on the moduli space $\mathcal{M}_{d}$ of dynamical systems, referring the reader to \cite[Chapter~4]{ads} for more details.  Let
$\Rat_{d}$ be the set of rational functions $\varphi: \PP^{1} \to \PP^{1}$, which is naturally identified with an open subset of $\PP^{2d+1}$ by simply taking the
coefficients of the numerator and denominator of $\varphi$ and eliminating the locus where the numerator and denominator have a common root.  Then $\Rat_{d}$ is a
variety defined over $\QQ$.  Since $\operatorname{Aut} (\PP^1)\cong \PGL_{2}$ acts naturally on rational functions by conjugation, representing a change of
coordinates, it makes sense to consider $\Rat_{d}$ up to this action by $\PGL_{2}$, prompting one to define
\[\mathcal{M}_{d} = \Rat_{d}/\PGL_{2}.\]  In \cite[Section 4.4]{ads}, it is
shown that $\mathcal{M}_{d}$ is an algebraic variety defined over $\ZZ$, and moreover, that if $K$ is an algebraically closed field, then $\mathcal{M}_{d}(K)$
consists precisely of the orbits of $\Rat_{d}$ under the conjugation action of $\PGL_{2}(K)$.

\begin{defn} The map $\varphi: \PP^{1} \to \PP^{1}$ is called a \emph{Latt\`{e}s map} if there is an elliptic curve $E$, a morphism $\alpha:E\to E$, and a finite
separable map $\pi$ such that the following diagram commutes:
\begin{equation*}
  \xymatrix@R+2em@C+2em{
  E \ar[r]^{\alpha} \ar[d]_{\pi} & E \ar[d]^{\pi} \\
  \PP^{1} \ar[r]_{\varphi} & \PP^{1}
  }
\end{equation*}
We say that $\varphi$ is a \emph{flexible Latt\`{e}s map} if it is obtained by taking $\pi$ to be the usual double-cover and taking $\alpha(P)=[m]P+T$, where $[m]$ is multiplication by an integer $m > 1$ and $T$ is a $2$-torsion element of $E$.
\end{defn}

Latt\`{e}s maps are treated more completely in \cite[Sections~6.4 and~6.5]{ads}, where it is shown that non-isomorphic elliptic curves yield non-conjugate Latt\`es maps \cite[Theorem~6.46]{ads}, and that Latt\`es maps are PCF \cite[Proposition~6.45]{ads}. Thus if we fix $m$, take $T$ to be the identity, and let $E$ vary, the corresponding flexible Latt\`es maps descending from $\alpha(P) = [m]P$ give a curve in $\mathcal{M}_{m^{2}}$ that consists of PCF maps.

The dimension of $\Rat_{d}$ is $2d+1$, and that of $\mathcal{M}_{d}$ is $2d-2$, since $\dim\PGL_{2} = 3$ and the $\PGL_{2}$-automorphism group of each $\varphi \in
\Rat_{d}$ is finite \cite{levy, tepper}. Since a rational function of degree $d$ has $2d-2$ critical points, we expect that each set of critical orbit relations $\varphi^{m_i}(\zeta_i) = \varphi^{n_i}(\zeta_i)$ on the orbits of the critical points $\zeta_i$ of $f$, with $m_i \neq n_i$, will have only finitely many solutions.   As it turns out, all of the flexible Latt\`es maps in a given family have the same post-critical behavior, and thus they form a
counterexample to this \emph{a priori} expectation. However, it follows from a deep result of Thurston \cite{thurston} that these are the only exceptions in
$\mathcal{M}_{d}(\CC)$.  We refer to this result as Thurston rigidity, though Thurston's actual result is considerably more broad.

\begin{theorem}[\textbf{Thurston rigidity}] \label{rigidity}
Apart from the flexible Latt\`es maps, there are only finitely many 
conjugacy classes of rational maps in $\mathcal{M}_{d}(\CC)$ satisfying a given set of critical orbit relations.  Moreover, all such conjugacy classes
have a representative defined over $\overline{\QQ}$.
\end{theorem}

Brezin \emph{et al.} \cite[Corollary~3.7]{pilgrim} explain how Thurston's result in \cite{thurston} implies Theorem~\ref{rigidity}.

Another main ingredient in the derivation of many of our corollaries is a celebrated result of McMullen \cite{mcmullen}.  Let $\Lambda_{n}: \mathcal{M}_{d} \to
\AA^{k_{n}}$ denote the \emph{multiplier spectrum map} \label{specm}, i.e., the map sending $\varphi$ to the unordered set of multipliers of its period-$n$ cycles (more precisely, to
the elementary symmetric functions in the multipliers).
\begin{theorem}[\textbf{McMullen}] \label{th:mcmullen}
Fix $d \geq 2$.  For sufficiently large $n$ the map
\begin{equation} \label{specmap}
\Lambda_{1} \times \cdots \times \Lambda_{n}: \mathcal{M}_{d}(\CC) \to \AA^{k_{1} + \cdots + k_{n}}(\CC)
\end{equation}
is finite-to-one away from the flexible Latt\`{e}s curves.
\end{theorem}
One can compute the multiplier spectra corresponding to Latt\`{e}s maps in a given family \cite[Proposition 6.52]{ads}, and see that these families are isospectral.  Thus, the map in \eqref{specmap} compresses each such family down to a single point.  The fact that the flexible Latt\`{e}s maps truly are exceptional in McMullen's theorem is crucial for the proof of Theorem~\ref{th:pcfmain}, since $\mathcal{M}_{d}(\CC)$ contains infinitely many flexible Latt\`{e}s maps defined over any given number field, as long as $d$ is a perfect square.

The last preliminary notion we require is the standard Weil height on $\overline{\QQ}$.
\begin{defn} \label{def:height}
The \emph{(logarithmic) Weil height}, or simply the \emph{height},
of an algebraic number $\alpha \in K$, where $K$ is a finite extension of $\QQ$, is defined by
\begin{equation}\label{eq:heightdef}
h(\alpha) = \sum_{v\in M_{K}}\frac{[K_{v}:\QQ_{v}]}{[K:\QQ]}\log\max\{1, |\alpha|_{v}\},
\end{equation}
where $M_{K}$ denotes the set of absolute values of $K$, normalized in the standard way (i.e. if $v | p$ then $|p|_v = 1/p$ and if $v | \infty$ then $v$ is the standard absolute value on $\CC$), and $K_v$ denotes the completion of $K$ with respect to the absolute value $v$.
\end{defn}
It is routine to check that the quantity \eqref{eq:heightdef}
does not depend on the field $K$, and hence~\eqref{eq:heightdef}
gives a well-defined function
$h:\overline{\QQ}\to\RR$.
The definition for heights over function fields is completely analogous.
We refer the reader to \cite[Chapter~3]{ads} for more details on Weil heights.

Northcott's
fundamental result in arithmetic geometry~\cite{northcott}
states that there are only finitely many algebraic numbers of bounded height and bounded degree. In other words,
for every pair of non-negative integers $A$ and $B$, there are only finitely many values $\alpha \in \overline{\QQ}$ satisfying both $$h(\alpha)\leq A\quad\text{ and
}\quad[\QQ(\alpha):\QQ] \leq B.$$
(The analogous result holds for function fields if and only if the field of constants is finite.) Moreover, this finite set of
points is effectively computable, since a bound on the height of an algebraic number yields a bound on the size of the coefficients of its minimal polynomial over $\ZZ$.  Finally, a
classical result of Kronecker implies that, over number fields, $h(\alpha)=0$ precisely if $\alpha=0$ or $\alpha$ is a root of unity (in dynamical terminology: precisely if $\alpha$ is preperiodic for $z\mapsto z^2$).  In the function field
setting, the condition $h(\alpha)=0$ is equivalent to $\alpha$ being a constant.

In order to justify our later focus on non-archimedean dynamics, rather than complex or global dynamics, we will now show how Theorem~\ref{th:pcfmain} follows from Theorem~\ref{th:main}. The remainder of the paper will be devoted to non-archimedean considerations related to the proof of Theorem~\ref{th:main}.
\begin{proof}[Proof of Theorem~\ref{th:pcfmain} and Corollary~\ref{cor:quadratic}]
Let $\varphi(z)\in\overline{\QQ}(z)$ be a PCF function of degree $d\geq 2$, and let $\lambda$ be the multiplier of a fixed point $\gamma$ of $\varphi$. We first wish to show that $h(\lambda)$ is bounded.  If $\lambda=0$, then $h(\lambda)=0$, and hence we may assume that $\lambda\neq 0$.

Let $K/\QQ$ be a finite extension containing the coefficients of $\varphi$, as well as its fixed points, and let $M_K$ denote the set of places of $K$.
For each $v\in M_K$, let $\CC_v$ denote the completion of the algebraic closure of the $v$-adic completion of $K$.  The key
observation is that there can be no critical point $\zeta\in\CC_v$ of $\varphi$ that is strictly attracted to a fixed  point of $\varphi$.  If that were the case, then
 that critical point would have an infinite forward orbit,
contradicting our assumption about $\varphi$.

If $v$ is archimedean, and therefore $\CC_v=\CC$,
then we may apply the above-mentioned result of Fatou that every attracting cycle attracts an infinite critical orbit, unless it is super-attracting. From this conclude that, since no infinite critical orbit exists to be attracted to $\gamma$, we must have
\begin{equation}\label{archbound}|\lambda|_v \geq 1.\end{equation}
If $v$ is non-archimedean, then $v$ extends some $p$-adic absolute value
on $\QQ$, and we write $v\mid p$.
In this case, we may similarly apply Theorem \ref{th:main} to
$\varphi$ over $\CC_v$, and thus obtain that $$|\lambda|_v \geq \epsilon_{p, d}.$$
We now invoke the standard fact \cite[Proposition~3.2]{ads} that for any prime
$p$ we have $$\sum_{\substack{v\in M_K\\ v\mid p}}\frac{[K_{v}:\QQ_{v}]}{[K:\QQ]}=1,$$ to obtain from Definition~\ref{def:height}
\begin{align}
h(\lambda) = h(\lambda^{-1}) & = \sum_{v\in M_{K}}\frac{[K_{v}:\QQ_{v}]}{[K:\QQ]}\log\max\{1, |\lambda^{-1}|_{v}\} \notag \\
& \leq \sum_{p\text{ prime}}\sum_{\substack{v\in M_K\\ v\mid p}}\frac{[K_{v}:\QQ_{v}]}{[K:\QQ]}\log \epsilon_{p, d}^{-1}
= \sum_{p \leq d} \log \epsilon_{p, d}^{-1},  \label{eqtwo}
\end{align}
where the first equality is a standard result of the product formula
(see \cite[Proposition~3.3]{ads}),
the inequality comes from discarding the archimedean places in light
of \eqref{archbound}, and the final equality is because
$\epsilon_{p,d}=1$ for all $p>d$. In the case $p=d=2$, the formula in Theorem~\ref{th:main} gives
\[\epsilon_{2, 2}=\min\left\{|m|^2:1\leq m\leq 2\right\}=\frac{1}{4},\] and so from the estimate \eqref{eqtwo} we immediately obtain Corollary~\ref{cor:quadratic}.

 In general, note that the bound in \eqref{eqtwo} is finite and depends only on $d$.
Thus if $\lambda$ is a multiplier of a fixed point of the PCF map $\varphi(z)\in\overline{\QQ}(z)$, we have a bound for $h(\lambda)$ which depends only on $d$. Applying this to $\varphi^n(z)$, for any $n$, we see that if $\lambda$ is the multiplier of a cycle of period $n$ for $\varphi$, then $h(\lambda)$ is bounded in terms of $d$ and $n$.

It follows from standard estimates that the elementary symmetric functions in the multipliers of $n$-cycles of a PCF map will also be of bounded height. In particular, if $\Lambda_n:\mathcal{M}_d\to\AA^{k_n}$ is the morphism taking a rational function to the elementary symmetric functions in the multipliers of its $n$-cycles, then there is an $N_{d, n}$ such that $h_{\PP^{k_n}}(\Lambda_n(\varphi))\leq N_{d, n}$ whenever $\varphi\in\mathcal{M}_d(\overline{\QQ})$ is PCF. Now, by Theorem~\ref{th:mcmullen} there exists an $n$ such that the map \[\Lambda=\Lambda_1\times \cdots\times\Lambda_n:\mathcal{M}_d\to\PP^{k_1}\times\cdots\times\PP^{k_n}\] is finite away from the set $\mathcal{L}$ of Latt\`{e}s maps.  Fixing such an $n$, we  define $h_{\mathcal{M}_d\setminus\mathcal{L}}$ to be the pull-back to $\mathcal{M}_d\setminus\mathcal{L}$ of the usual height on $\PP^{k_1}\times\cdots\times\PP^{k_n}$ by $\Lambda$. From the discussion above, we know that Latt\`{e}s maps live in a set which is  bounded with respect to this height. The second claim in Theorem~\ref{th:pcfmain} now follows from the standard Northcott property of the height on $\PP^{k_1}\times\cdots\times \PP^{k_n}$, since the map $\Lambda$ is defined over $\QQ$.
\end{proof}

Note that we can be quite explicit about the bound on $h(\lambda)$.
Recalling from Theorem~\ref{th:main} that
\[\epsilon_{p, d}=\min\left\{|m|^d_p:1\leq m\leq d\right\},\]
our height bound for multipliers of fixed points of PCF maps of degree
$d\geq 2$ becomes
\[h(\lambda)\leq d\sum_{p\leq d} \log\max\{|m|_p^{-1}:1\leq m\leq d\}=d\sum_{n\leq d}\Lambda(n),\]
where $\Lambda$ now denotes the von Mangoldt function \cite[Section~2.8]{aposotol}.  Since the prime number theorem is equivalent to the fact that $\sum_{n\leq x}\Lambda(n)$ is asymptotic to $x$, our upper bound is asymptotic to $d^2$, although the quality of the error term depends on which conjectures of analytic number theory one is prepared to adopt.

It remains to prove Theorems~\ref{th:main} and~\ref{th:nondyn}, which
we do in Section~\ref{proof}, and Theorem~\ref{introper}, which
we do in Section~\ref{perpts}.


\section{Background on non-archimedean analysis}
\label{sec:nonarch}
In this section, we summarize the definitions and results
on non-archimedean analysis, and especially Berkovich spaces,
that we will need to prove Theorem~\ref{th:main}.

Fix an algebraically closed non-archimedean field $K$ with
absolute value $|\cdot|$ as in Theorem~\ref{th:main}.
By an \emph{open disk} in $\PP^1(K)$ we mean
an open disk \[D(a,r)=\{x\in K : |x-a|<r\}\] in $K$
or the complement $\PP^1(K)\setminus \Dbar(a,r)$
of a closed disk in $K$.  Similarly, a \emph{closed disk}
in $\PP^1(K)$ is either a closed disk in $K$ or the complement
of an open disk in $K$.  In either case, we say the disk
is \emph{rational} if the radius $r>0$ lies in $|K^{\times}|$.

A \emph{closed} (respectively, \emph{open})
\emph{connected affinoid} is the intersection of finitely
many closed (respectively, open) disks in $\PP^1(K)$.
We say the affinoid is \emph{rational} if all the disks
in the intersection are rational.

Let $U\subseteq\PP^1(K)$ be a connected affinoid, and let
$h\in K(z)$ be a rational function of degree $d\geq 1$.
Then $h^{-1}(U)$ is the disjoint union of $1\leq \ell\leq d$
connected affinoids $V_1,\ldots, V_{\ell}$, where
for each $i=1,\ldots, \ell$, there is an integer $1\leq m_i\leq d$
such that $h$ maps $V_i$ everywhere $m_i$-to-one onto $U$.
Moreover, $m_1 + \cdots + m_{\ell}= d$.
(See \cite[Proposition~2.5.3]{rbthesis} or
\cite[Proposition~2.6]{jrl}, for example.)
The connected affinoids $V_1,\ldots, V_{\ell}$
are called the \emph{components} of $h^{-1}(U)$.
If $U$
is closed (respectively, open, rational), then every component $V_i$
of $h^{-1}(U)$ is also closed (respectively, open, rational).
Moreover, if $U$ is a disk and $m_i=1$, then $V_i$ is also a disk.
In addition, if $U\subseteq K$ is a finite disk and $h$ is
a polynomial, then each $V_i$ is also a disk.
For further information on affinoids and rigid analysis,
see \cite{BGR,conrad,FvP}.

The Berkovich projective line $\PBerk$ over $K$ is a certain space
of multiplicative seminorms on $K$-algebras.  It contains $\PK$
as a subspace but is path-connected, compact, and Hausdorff.
The Berkovich affine line over $K$ is defined to be
$\ABerk = \PBerk\setminus\{\infty\}$,
the Berkovich hyperbolic space is
$\HBerk=\PBerk\setminus\PK$.
The full definition of $\PBerk$ is rather involved;
for details, the interested reader may consult
Berkovich's original presentation in \cite{berk},
the thorough exposition in \cite{bakrum}, or the summaries
in \cite[Section~4]{benedetto}, 
\cite[Sections~6.1--6.3]{benedAZ},
\cite[Sections~2.1--2.2]{FRL},
and \cite[Section~5.10]{ads}.
Still, we present a general description here, without proofs.

Each point $\zeta\in\ABerk$
is associated to a multiplicative seminorm
on $K[z]$ extending $|\cdot|$, and we
denote this seminorm by $\|\cdot\|_{\zeta}$.
As a typical example, for each closed disk
$\Dbar(a, r)\subseteq K$ of finite radius $r>0$,
there is a corresponding point $\zeta(a, r)$ in Berkovich
space defined by
\[\|f\|_{\zeta(a, r)}=\sup\{|f(z)|: z\in \Dbar(a, r)\}.\]
Equivalently, if we write $f(z) = \sum_i c_i (z-a)^i$,
we have $\|f\|_{\zeta(a, r)}=\sup\{|c_i|r^i : i\geq 0\}$.
The point $\zeta(a,r)$ is said to be of type~II if $r\in |K^{\times}|$,
or of type~III if $r\in (0,\infty)\setminus |K^{\times}|$.
In other words, type~II Berkovich points correspond to rational
closed disks in $K$, and type~III Berkovich points correspond to
irrational closed disks in $K$.

Meanwhile, each $x\in K$ induces a seminorm $\|\cdot\|_x$
defined by $\|f\|_x = |f(x)|$.  Such seminorms are the type~I
points of $\ABerk$, and the mapping $K\to\ABerk$ by
$x\mapsto \|\cdot\|_x$ is a topological embedding.
There are also points of type~IV, corresponding to decreasing
chains of disks with empty intersection, but such points will
not concern us here.  The hyperbolic space $\HBerk$ consists
of the points of types~II, III, and~IV.

Any seminorm $\|\cdot\|_{\zeta}$ in $\HBerk$ is actually a norm on $K[z]$ and therefore may be extended
to $K(z)$ by setting
$\|f/g\|_{\zeta} = \|f\|_{\zeta} / \|g\|_{\zeta}$, which
is independent of the choice of polynomials $f,g\in K[z]$
representing the rational function $f/g$.
The same definition also makes sense at type~I points $x\in K$,
provided we allow
$\|f/g\|_{x}$ to take on the value $\infty$ if $g(x)=0$.
Meanwhile, we may define $\|\cdot\|_{\infty}$
at the one remaining point $\infty\in\PBerk$ by
setting $\|h(z)\|_{\infty} = \|h(1/z)\|_{0}$ for any
$h\in K(z)$.

Any point $\zeta\in\PBerk$ has a
radius $\rad(\zeta)\in [0,\infty]$, defined by
$$\rad(\zeta) = \inf\{ \|z-a\|_{\zeta} : a\in K\}.$$
In particular,
if $\zeta=\zeta(a,r)$ is a point of type~II or~III,
corresponding to the closed disk $\Dbar(a,r)$,
then $\rad(\zeta)=r$.
Meanwhile,
$\rad(a)=0$ for each type~I point $a\in K$,
and $\rad(\infty)=\infty$.
The reader should be warned that
the function $\rad:\PBerk\to [0,\infty]$
is not continuous, but only upper semicontinuous.
However, $\rad$ is continuous on line segments in $\PBerk$.
Specifically,
fix any point $a\in K$, and let $L_a$ be the line segment in $\PBerk$
from $a$ to $\infty$.  Then $\rad$ is continuous on $L_a$,
and in fact $\log\circ\rad:L_a \to [-\infty,\infty]$ is a homeomorphism,
with inverse $t\mapsto \zeta(a,\exp(t))$.

Given $\zeta\in\PBerk$, each connected component of
$\PBerk\setminus\{\zeta\}$ is called a \emph{tangent direction}
at $\zeta$.  If $\zeta=\zeta(a,r)$ is of type~III,
then it has two tangent directions: the component containing $\infty$
and the component containing $a$.  On the other hand, if
$\zeta=\zeta(a,|c|)$ is of type~II, then it has infinitely many
tangent directions: one containing $\infty$, and one containing
each point $a+cu$, as $u\in K$ ranges over a set of representatives
of the residue field of $K$.  A point of type~I or~IV has only
one tangent direction.

For any fixed nonzero rational function $h\in K(z)$,
the function $\zeta\mapsto \|h\|_{\zeta}$
from $\PBerk$ to $[0,\infty]$ is continuous.
(This statement is essentially the definition of the topology on $\PBerk$.)
For any fixed $a\in K$,
the graph of the function $\log t\mapsto \log\|h\|_{\zeta(a,t)}$ is
called the \emph{valuation polygon}
or \emph{Newton copolygon} of the rational function $h(z-a)$.
This function,
which is the composition of the homeomorphism
$\zeta(a,\exp(\cdot)):[-\infty,\infty]\to L_a$ with
the map $\zeta\mapsto\|h\|_{\zeta}$,
is continuous and piecewise linear, and the slope
of each of its segments is necessarily an integer.
More precisely,
for any point $\zeta=\zeta(a,r)$
of type~II or~III and any $b\in\PP^1(K)$, set
$N_a^+(h,\zeta(a,r),b)$ to be the nonnegative integer
\[N_a^+(h,\zeta(a,r),b) = \#\{z\in\Dbar(a, r):h(z)=b\},\]
counted with multiplicity,
and define $N_a^-(h, \zeta(a, r), b)$
similarly relative to the open disk $D(a, r)$.
Then the Newton copolygon function
$\log t \mapsto \log \|h\|_{\zeta(a,t)}$
from $[-\infty,\infty]$ to $[-\infty,\infty]$
has integer slope to the left of $\log r$ given by
\begin{equation}
\label{eq:leftslope}
N_a^-(h,\zeta(a,r),0)-N_a^-(h,\zeta(a,r),\infty),
\end{equation}
and to the  right of $\log r$ given by
\begin{equation}
\label{eq:rightslope}
N_a^+(h,\zeta(a,r),0)-N_a^+(h,\zeta(a,r),\infty).
\end{equation}
This integer is precisely the
Weierstrass degree (i.e., the degree of the term
of maximal absolute value) of the Laurent series expansion
$h(z) =\sum_{i\in\ZZ} c_i (z-a)^i$ on the annulus
$X^- = D(a,r)\setminus \Dbar(a,r-\epsilon)$
or
$X^+ = D(a,r+\epsilon)\setminus \Dbar(a,r)$,
respectively,
for sufficiently small $\epsilon>0$.
For further details on valuation polygons,
see the foundational work in \cite{robba}, as well as the
expositions in \cite[VI.1.6,VI.3.3]{robert} and
\cite[Section~6]{benedetto}.
For example, a proof of the piecewise linearity statement
above may be found in \cite[Section~3]{robba} or
\cite[VI.1.6ff]{robert}, albeit not phrased in the language
of Berkovich spaces.

Any rational function $\phi\in K(z)$ induces a continuous function
$\phi:\PBerk\to\PBerk$, where for each $\zeta\in\PBerk$, the image
$\phi(\zeta)$ is the seminorm defined by
$$\|h\|_{\phi(\zeta)} = \|h\circ\phi\|_{\zeta}
\qquad
\text{for all } h\in K(z).$$
It is easy to check that for type~I points $\zeta=x\in\PK$, this
definition of $\phi(x)$ coincides with the usual action of
$\phi$ on $\PK$.  If $\phi$ is nonconstant, then for any
$\zeta\in\PBerk$, its image $\phi(\zeta)$ is a point of the same type.
If $\zeta=\zeta(a,r)\in\PBerk$ is of type~II or~III,
then for each tangent direction $\vec{v}$ at $\zeta$, $\phi$
induces a tangent direction $\phi_*(\vec{v})$ at $\phi(\zeta)$,
as follows.  If $\vec{v}$ is the tangent direction containing $\infty$,
then for all sufficiently small $\epsilon>0$, the image
$\phi(X)$ of the annulus $X=D(a,r+\epsilon)\setminus\Dbar(a,r)$
is contained in a single tangent direction $\vec{w}$ at $\phi(\zeta)$.
Otherwise, if $\vec{v}$ is the tangent direction containing
$b\in\Dbar(a,r)$, then
for all sufficiently small $\epsilon>0$, the image
$\phi(X)$ of the annulus $X=D(b,r)\setminus\Dbar(b,r-\epsilon)$
is contained in a single tangent direction $\vec{w}$ at $\phi(\zeta)$.
In either case, the image $\phi(X)$ is an annulus of the form
either $D(c,s+\delta)\setminus\Dbar(c,s)$
or $D(c,s)\setminus\Dbar(c,s-\delta)$,
the image point $\phi(\zeta)$ may be written as $\phi(\zeta)=\zeta(c,s)$,
and the image direction is defined to be $\phi_*(\vec{v})=\vec{w}$,
the direction at $\phi(\zeta)$ containing $\phi(X)$.

Just as $\phi$ maps points of $\PK$ to one another with multiplicity,
$\phi_*$ also maps tangent directions to one another with multiplicity.
Indeed,
the annulus $X$ in the previous paragraph maps to its image with
some multiplicity $1\leq m\leq \deg\phi$ that is independent of
the sufficiently small $\epsilon>0$.  We define
the multiplicity $\deg_{\zeta,\vec{v}}(\phi)$ of $\phi$
in the direction $\vec{v}$ at $\zeta$ to be this integer $m$.
If $\phi(\zeta)=\zeta(0,s)$ and $\phi_*(\vec{v})$ is the direction
at $\zeta(0,s)$ of either $0$ or $\infty$, then
$\deg_{\zeta,\vec{v}}(\phi) = |m|_{\infty}$, where
$m$ is the Weierstrass degree (necessarily nonzero in this case)
on the appropriate annulus from
either \eqref{eq:leftslope} or~\eqref{eq:rightslope},
respectively, and $|\cdot|_{\infty}$ denotes the
(usual) archimedean absolute value.
Readers familiar with $\PBerk$ will recognize that
the integer $m=\deg_{\zeta,\vec{v}}(\phi)$ is denoted
$m_{\phi}(\zeta,\vec{v})$ in \cite[Section~9.1]{bakrum};
previously, in \cite[Lemme~2.1]{jrl}, with $\vec{v}$ denoted
$\mathcal{P}$ and referred to as a ``bout'', or end,
it had been denoted $\deg_{\phi}(\mathcal{P})$.
It is less than or equal to the multiplicity or local degree
of $\phi$ at $\zeta$,
denoted by $m_{\phi}(\zeta)$ in \cite{bakrum},
and by $\deg_{\phi}(\zeta)$ in \cite{jrl,FRL}.

The multiplicities
$\deg_{\zeta,\vec{v}}(\phi)$ satisfy the following useful properties.

\begin{lemma} \label{lem:mult}
Let $\phi,\psi\in K(z)$ be rational functions,
let $\zeta\in\HBerk$, and let $\vec{v}$ be a direction at $\zeta$.  Then
$$\deg_{\zeta,\vec{v}}(\psi\circ\phi)
= \deg_{\phi(\zeta),\phi_*(\vec{v})} (\psi) \cdot
\deg_{\zeta,\vec{v}}(\phi).$$
\end{lemma}

\begin{proof}
Although this statement is true for points of all types, we will only
use or prove it for types~II and~III.
Let $X$ be a sufficiently small annulus abutting $\zeta$,
and let $Y=\phi(X)$.  Then because $\phi:X\to Y$ has degree
$\deg_{\zeta,\vec{v}}(\phi)$, and
$\psi:Y\to\psi(Y)$ has degree
$\deg_{\phi(\zeta),\phi_*(\vec{v})}(\psi)$, the desired equality
is immediate.
\end{proof}

\begin{lemma} \label{lem:multcompute}
Let $\phi\in K(z)$, let $\zeta\in\HBerk$, and let $a\in K$.
Let $\vec{v}$ be the direction at $\zeta$ containing $a$,
let $\vec{w}$ be the direction at $\zeta$ containing $\infty$,
and assume that $\vec{v}\neq\vec{w}$.
\begin{list}{\rm \alph{bean}.}{\usecounter{bean}}
\item
If $\phi_*(\vec{v})$ is the direction at $\phi(\zeta)$
containing $0$, then
$$\deg_{\zeta,\vec{v}}\phi = N_a^-(\phi,\zeta,0)-N_a^-(\phi,\zeta,\infty)\geq 1.$$
\item
If $\phi_*(\vec{w})$ is the direction at $\phi(\zeta)$
containing $0$, then
$$\deg_{\zeta,\vec{w}}\phi = -N_a^+(\phi,\zeta,0)+N_a^+(\phi,\zeta,\infty)\geq 1.$$
\end{list}
\end{lemma}

\begin{proof}
This is simply the alternate characterization of the tangent
direction multiplicities in terms of
the Weierstrass degrees~\eqref{eq:leftslope}
and~\eqref{eq:rightslope}.
\end{proof}

We close this section by describing another function on $\HBerk$,
similar to one introduced in \cite[Section~6]{benedetto}.
Fix a nonzero rational function $\phi\in K(z)\setminus\{0\}$.
Define the \emph{distortion} of $\phi$ to be the real-valued function
$\delta(\phi,\cdot)$ on $\HBerk$ given by
\begin{equation}
\label{eq:hdef}
\delta(\phi,\zeta) = \log\rad(\zeta) + \log\|\phi'\|_{\zeta}
-\log\|\phi\|_{\zeta}.
\end{equation}

\begin{lemma} \label{lem:hbound}
Let $\phi\in K(z)\setminus\{0\}$,
let $a\in K$, let $r>0$, and set $\zeta=\zeta(a,r)$.
Then
\begin{equation}
\label{eq:hbound}
\log |N_a^{\pm}(\phi,\zeta,0)-N_a^{\pm}(\phi,\zeta,\infty)|
\leq \delta(\phi,\zeta) \leq 0,
\end{equation}
where $|\cdot|$ denotes, as always, the absolute value on $K$.
\end{lemma}

\begin{proof}
We will prove the upper bound for $N_a^-$; the proof for $N_a^+$ is similar.
After a change of coordinates on the domain, we may assume that
$a=0$
and expand $\phi(z)$ as a Laurent series $\sum_{i\in\ZZ} c_i z^i$
on a sufficiently small annulus
$X=\{x\in K : r-\epsilon<|x|<r\}$.
Setting
$m = N_a^-(\phi,\zeta(0,r),0)-N_a^-(\phi,\zeta(0,r),\infty)$,
we have
$|\phi(x)| = |c_m x^m| > |c_i x^i|$ for all
$x\in X$
and $i\in\ZZ\setminus\{m\}$.
Thus, for all $t\in (r-\epsilon,r)$, we have
$$ t \|\phi'\|_{\zeta(0,t)}
= t \sup_{i\in\ZZ} |ic_i| t^{i-1} \geq |mc_m| t^m
= |m| \cdot \|\phi\|_{\zeta(0,t)},$$
and hence $\delta(\phi,\zeta(0,t))\geq \log|m|$.
In addition,
$$ t \|\phi'\|_{\zeta(0,t)}
= t \sup_{i\in\ZZ} |ic_i| t^{i-1}
\leq t \sup_{i\in\ZZ} |c_i| t^{i-1}
= \sup_{i\in\ZZ} |c_i| t^i = \|\phi\|_{\zeta(0,t)},$$
and hence $\delta(\phi,\zeta(0,t))\leq 0$.
(See also \cite[inequality~(6.5)]{benedetto}.)
The desired bounds are now immediate by
taking limits as $t\nearrow r$,
by the continuity of $\rad$,
$\zeta\mapsto\|\phi\|_{\zeta}$,
and $\zeta\mapsto\|\phi'\|_{\zeta}$
on the line segment $L_0\subseteq\PBerk$.
\end{proof}

\begin{remark}\label{rem:deltabound}Combining the previous two lemmas, note that under the hypotheses of the first part of Lemma~\ref{lem:multcompute}, the distortion $\delta(\phi,\zeta)$ is bounded below by $\log|\deg_{\zeta,\vec{v}}(\phi)|$, which is finite if $\charact K = 0$ or $\charact K > \deg \phi$. Specifically, under the same hypotheses, $\delta(\phi,\zeta)$ is bounded below by $\min\{\log|m| : 1\leq m \leq \deg\phi\}$.  This fact will be essential in our proof of Theorem~\ref{th:main}.\end{remark}

\section{Proof of Theorems~\ref{th:main} and~\ref{th:nondyn}}\label{proof}

In this section, we establish Theorem~\ref{th:main}, using an
argument on the Berkovich analytic space $\PBerk$ associated to $\PP^1_K$.
The outline of the proof is as follows.  First, there is a maximal open
disk $U$ containing the attracting fixed point $\gamma$,
with the property that $\phi$ contracts all distances in $U$ by a factor
of exactly $|\phi'(\gamma)|$.  In particular, all points in $U\setminus\{\gamma\}$ are strictly
attracted to $\gamma$, and hence it suffices to show that
$U\setminus \{\gamma\}$ contains a critical value.
We define a certain function $G:\PBerk\to [-\infty,\infty]$ that
involves the Newton copolygons of $\phi$ and $\phi'$; the slopes of
$G$ count the numbers of zeros, poles, and critical points of $\phi$
inside various disks.  By controlling the growth of $G$ and then using
these slopes to count carefully, we will be able to show that there
are more critical points in $\phi^{-1}(U)$ than can map to $\gamma$.

More precisely, Theorem~\ref{th:main} will be a consequence of the following.

\begin{theorem}\label{expbound} Let $K$ be a field satisfying the hypotheses of Theorem \ref{th:main}, and let $\phi\in K(z)$ be a rational function of degree $d\geq 2$.  Suppose that $\phi$ has a fixed point $\gamma$ with multiplier $\lambda \neq 0$ satisfying
\begin{equation}\label{ass}
0<|\lambda| < |\deg_{\zeta,\vec{v}}\phi|^d
\end{equation}
for all $\zeta\in\PBerk$ and all directions $\vec{v}$ at $\zeta$. Then there is a disk $U\subseteq\PK$ containing $\gamma$ such that $\phi$ maps $U$ into itself injectively, all points of $U$ are attracted to $\gamma$ under iteration, and $U$ contains a critical value that is strictly attracted to $\gamma$.
\end{theorem}

\begin{proof} We will in fact only consider $\deg_{\zeta,\vec{v}}\phi$ at a few points $\zeta$ in the immediate attracting basin of $\gamma$.

After a change of variables, we may make the following three assumptions: that $\gamma=0$, that $\phi(\infty)=\infty$, and that the minimum absolute value of a non-zero root or pole of $\phi$ is 1. The second assumption is legitimate because $\phi$ has at least one other fixed point, and the third assumption is legitimate because we may then conjugate by a map of the form $z\mapsto cz$. Define $U=D(0,1)$. Note that because $\phi$ has no poles in $U$ and only the simple zero at $z=0$, we have $|\phi(z)|=|\lambda z|$ for all $z\in U$. Thus $\phi(U)=D(0,|\lambda|)$, which is a proper subset of $U$.

Let $V$ be the connected component of $\phi^{-1}(U)$ containing $U$, so that
$U$ is a proper subset of $V$, and $\phi$ is an $m$-to-$1$ map of $V$ onto $U$, for some integer $1\leq m\leq d$.  We may write $V=D(0,R)\setminus (W_1\cup \cdots\cup W_n)$, where
$R>1$ and each $W_i = \Dbar(b_i,s_i)$ is a rational closed disk contained in $D(0,R)$ and not intersecting $D(0,1)$. We have $\phi(\zeta(0,R))=\zeta(0,1)$ and
$\phi(\zeta(b_i,r_i)) = \zeta(0,1)$ for each $i=1,\ldots, n$.

Define $G:\HBerk\to \RR$ by
$$G(\zeta) = m \delta(\phi,\zeta) + \log\|\phi\|_{\zeta} =
  m\log\rad(\zeta) + m\log\|\phi'\|_{\zeta} + (1-m)\log\|\phi\|_{\zeta}.$$
Note that $G$ is continuous along any line segment in $\HBerk$.

First, since $\|\varphi\|_{\zeta(0, 1)} = |\lambda|$ and $N_0^-(\phi, \zeta(0, 1), 0) - N_0^-(\phi, \zeta(0, 1), \infty) = 1$, Lemmas~\ref{lem:multcompute} and~\ref{lem:hbound} imply that \[G(\zeta(0,1)) = m\cdot 0 + \log |\lambda| = \log|\lambda|.\] On the other hand, Lemma~\ref{lem:multcompute} also implies that 
$$N_0^-(\phi,\zeta(0,R),0) - N_0^-(\phi,\zeta(0,R),\infty) = \deg_{\zeta(0,R),\vec{v}}\phi$$
 where $\vec{v}$ is the tangent direction at $0$ and
$$N_{b_i}^+(\phi,\zeta(b_i,r_i),\infty) - N_{b_i}^+(\phi,\zeta(b_i,r_i),0) = \deg_{\zeta(b_i,r_i), \vec{w}}\phi$$
 for each $1 \leq i \leq n$ where $\vec{w}$ is the tangent direction at $\infty$. Then Lemmas~\ref{lem:multcompute} and~\ref{lem:hbound} and the assumption~\eqref{ass} combine to show that
\[G(\zeta(0,R)) = m\delta(\phi,\zeta(0,R)) \geq m\log|\deg_{\zeta(0,R),\vec{v}}\phi| > G(\zeta(0,1))\]
and
\[G(\zeta(b_i,r_i)) = m\delta(\phi,\zeta(b_i,r_i)) \geq m\log|\deg_{\zeta(b_i,r_i),\vec{w}}\phi| > G(\zeta(0,1)).\]

Consider the unique interval in $\HBerk$ from $\zeta(0,1)$ to
$\zeta(0,R)$, which we identify with the real interval $[0, \log R]$.
Since the function $\log r \mapsto G(\zeta(0,r))$ on $[0,\log R]$
is piecewise linear and continuous, and since
$G(\zeta(0,R))>G(\zeta(0,1))$, there must be some subinterval along
which $G$ is increasing.
On the other hand, the slope of this function
is precisely
\begin{equation}
\label{eq:localslope}
m [1+ N_0^+(\phi',\zeta,0) - N_0^+(\phi',\zeta,\infty)]
+ (1-m) [N_0^+(\phi,\zeta,0) - N_0^+(\phi,\zeta,\infty)],
\end{equation}
where $\zeta=\zeta(0,r)$, at each point $\log r$
at which the function is smooth.
Thus, the integer \eqref{eq:localslope} must be positive, and hence
at least $1$, for some $r\in [1, R)$;
let $S$ denote the infimum of all such $r$.
By right-continuity of $N_0^+$, it must be the case that
\begin{multline}
\label{eq:Sbound}
m [1 + N_0^+(\phi',\zeta(0,S),0) - N_0^+(\phi',\zeta(0,S),\infty)]
\\
+ (1-m) [N_0^+(\phi,\zeta(0,S),0) - N_0^+(\phi,\zeta(0,S),\infty)] \geq 1.
\end{multline}

If $G(\zeta(0, S))>G(\zeta(0, 1))$, then again there must be a
subinterval of $[1, S]$ along which $G$ is increasing, which
contradicts the definition of $S$.  Thus, we must have
\[G(\zeta(0,S))\leq G(\zeta(0,1))<G(\zeta(b_i,r_i))\]
for all $1\leq i\leq n$. Once
again, then, $G$ must increase along some subinterval of the interval
in $\PBerk$ running from $\zeta(0,S)$ to $\zeta(b_i,s_i)$. We discard
each index $i$ for which $b_i\not\in\Dbar(0,S)$ and suppose, without
loss of generality, that the remaining indices are $i=1,\ldots, k$.

For each $i$, we again identify the interval in $\HBerk$
from $\zeta(0,S)=\zeta(b_i,S)$ to $\zeta(b_i,r_i)$
with the real interval $[\log r_i, \log S]$; indeed, the former consists
of all Berkovich points of the form $\zeta(b_i,r)$ with $r_i\leq r\leq S$.
As before, because the continuous,
piecewise linear function $\log r \mapsto \log G(\zeta(b_i,r))$
is greater at $\log r_i$ than at $\log S$, there must be points $\log r$
at which the slope of this function is negative,
and hence at most $-1$.
Let $s_i$ be the supremum of all such $r$.
By the left-continuity of $N_{b_i}^-$, we deduce that
\begin{multline}
\label{eq:sibound}
m [1 + N_{b_i}^-(\phi',\zeta(b_i,s_i),0)
- N_{b_i}^-(\phi',\zeta(b_i,s_i),\infty)]
\\
+ (1-m) [N_{b_i}^-(\phi,\zeta(b_i,s_i),0)
- N_{b_i}^-(\phi,\zeta(b_i,s_i),\infty)] \leq -1.
\end{multline}
Let $W=\Dbar(0,S)\setminus[ D(b_1,s_1) \cup\cdots\cup D(b_k,s_k) ]$, and for any nonzero $h\in K(z)$ and $a\in\PP^1(K)$, let $N(h,W,a)$ be the number of
roots of $h(z)=a$ in $W$, counting multiplicity. Summing inequality~\eqref{eq:sibound} across $i=1,\ldots,k$ and subtracting from inequality~\eqref{eq:Sbound},
we have
$$m [(1-k) + N(\phi',W,0) - N(\phi',W,\infty)]
+ (1-m) [N(\phi,W,0) - N(\phi,W,\infty)] \geq 1+k.$$
However, $W\subseteq V$, and therefore $\phi(W)\subseteq
\phi(V)=U$. In particular, $\phi$ has no poles in $W$, and hence neither does $\phi'$.
Thus,
\begin{equation}
\label{eq:zerocrit}
m [(1-k) + N(\phi',W,0)] + (1-m) [N(\phi,W,0)] \geq 1+k.
\end{equation}

Let $M\geq 0$ denote the number of critical points in $W$ that are
\emph{not} zeros of $\phi$, counted with multiplicity.
Let $t$ denote the number of \emph{distinct} zeros of $\phi$ in $W$.
For each such zero $x\in W$ of $\phi$,
the order of vanishing of $\phi'$ at $x$ is one less than
the order of vanishing of $\phi$ at $x$,
since by hypothesis~\eqref{ass},
$\phi$ has no wildly ramified critical points.
Thus,
\begin{equation}
\label{eq:Mbound}
N(\phi',W,0) = M + N(\phi,W,0) - t,
\end{equation}
Incorporating equation~\eqref{eq:Mbound}
into \eqref{eq:zerocrit}, then,
\begin{equation}
\label{eq:Mbound2}
mM \geq 1 + k - m + mk + mt - N(\phi,W,0)
\geq (1+k)(1+m) + (t-3)m,
\end{equation}
where the second inequality is because $W\subseteq V$,
and hence $N(\phi,W,0)\leq m$.

We claim that $M>0$.  Note that $t\geq 1$ because of the simple
zero of $\phi$ at $0\in W$.  Thus,
if $k\geq 1$, our claim is immediate from inequality~\eqref{eq:Mbound2}.
On the other hand, if $k=0$, then $W=\Dbar(0,S)\supseteq\Dbar(0,1)$.
By the choice of coordinate at the start of this proof, there is
a zero or pole $y$ of $\phi$ with $|y|=1$; in particular,
$y\in W\setminus\{0\}$.  As we noted earlier, however, $\phi$
has no poles in $W$, and hence $y$ must be a zero of $\phi$.
Thus, $t\geq 2$, and again the claim follows from
inequality~\eqref{eq:Mbound2}.

Since $M >0$, there is a critical point $\alpha\in W$ such that
$\phi(\alpha)\in U\setminus\{0\}$.  However, recall that
$\phi:U\to U$ has the property that $|\phi(z)|=|\lambda z|$ for all
$z\in U$.  Thus, $\phi^n(\alpha)\to 0$ as $n\to \infty$, but
$\phi^n(\alpha)\neq 0$ for all $n\geq 0$, as desired.
\end{proof}


\begin{proof}
[Proof of Theorem~\ref{th:main}]
As described in Section~\ref{sec:nonarch}, we have
$1\leq \deg_{\zeta,\vec{v}}\phi \leq d$ for all $\zeta\in\PBerk$
and all directions $\vec{v}$ at $\zeta$.
Since we had $\epsilon=\min\{|m|^d : 1\leq m \leq d\}$,
inequality~\eqref{ass} holds for all $\zeta\in\PBerk$.
Thus, the conclusion of Theorem~\ref{th:main} is immediate
from Theorem~\ref{expbound}.
\end{proof}

\begin{proof}[Proof of Theorem~\ref{th:nondyn}]
If $\varphi'(a) = 0$, then the result is trivial, as $a$ is a suitable
critical point.  We may therefore assume that $\varphi'(a) \neq 0$.
We may also choose $c\in K$ with $|c|<\epsilon / |\phi'(a)|$, and
with $|c|$ arbitrarily close to $\epsilon / |\phi'(a)|$.
Recalling that $\varphi(a)\neq\infty$ by hypothesis, define
$h(z) = c(z-\phi(a)) + a$, and
$\psi(z) = h(\phi(z))$.
Then $\psi(\infty)=\infty$, $\psi(a)=a$, and $\psi'(a) = c\varphi'(a)$
satisfies $0<|\psi'(a)|<\epsilon$.

Applying Theorem~\ref{expbound} to $\psi$ produces a disk $U\subseteq\PK$
containing $a$ and a critical point $\beta$ such that
\begin{enumerate}
\item $\psi$ maps $U$ into itself injectively,
\item $\psi^n(x)\to a$ for all $x\in U$, and
\item $\psi(\beta)\in U$.
\end{enumerate}
By property~(2), we have $\infty\not\in U$, and therefore property~(1)
implies that $U\setminus\{a\}$ does not intersect
$\psi^{-1}(\{a,\infty\})=\phi^{-1}(\{\phi(a),\infty\})$.
In other words, $U$ is contained in the largest disk
$V\subseteq K$ containing $a$ but no other points of
$\phi^{-1}(\{\phi(a),\infty\})$.
By hypothesis, however, $V=D(a,r)$.

Meanwhile, $\phi$ and $\psi$ have the same critical points; in particular,
$\beta$ is a critical point of $\phi$.
Noting that the inverse function of $h$ is
$h^{-1}(z) = c^{-1}(z-a) + \phi(a)$, we compute
$$\phi(\beta) = h^{-1}(\psi(\beta)) \in h^{-1}(U)
\subseteq h^{-1}(D(a,r)) = D(\phi(a),|c|^{-1}r).$$
Taking the intersection of the disks $D(\phi(a),|c|^{-1}r)$ across
all $c\in K$ with $|c|<\epsilon/|\phi'(a)|$, and bearing in mind
that $\phi$ has only finitely many critical points, it follows
that there is a critical point $\beta$ satisfying
$\phi(\beta)\in \Dbar(\phi(a), \epsilon^{-1} r |\phi'(a)|)$,
as desired.
\end{proof}


\section{Sharpness of the Bound in Theorem \ref{th:main}}
\label{sec:sharp}

As we commented in the introduction, the bound $\epsilon_{p,d}$
of Theorem~\ref{th:main} can be improved for polynomials.

\begin{theorem}
\label{th:polys}
Let $K$, $|\cdot|$, $p$, and $d$ be as in Theorem~\ref{th:main},
and define
\[\epsilon=\epsilon_{p, d}^{\mathrm{poly}}=\min\{|m|^m:1\leq m\leq d\}
\geq \epsilon_{p,d}.\]
Let $\varphi(z) \in K[z]$ be a polynomial of degree $d$,
and let $\gamma$ be a fixed point of $\varphi$ satisfying
$$0 < |\varphi'(\gamma)| < \epsilon.$$
Then there is a critical point of $\varphi$ which is strictly attracted to
$\gamma$.
\end{theorem}

\begin{proof}We essentially follow the proof of Theorem~\ref{expbound}. In particular, we may assume that $\gamma=0$, and that $U=D(0,1)$ is the largest disk containing $0$ on which $\phi$ is injective.  Setting $V$ to be the connected component of $\phi^{-1}(U)$ containing $0$, we conclude that the mapping $\phi:V\to U$ is everywhere $m$-to-$1$, for some integer $1 < m \leq d$ (we cannot have $m = 1$ because then $\phi$ would be injective on $V \supsetneq U$). This time, however, we know that $V$ is a disk $D(0,R)$, since $\phi$ is a polynomial, and that $\phi$ has exactly $m$ zeros and no poles in $D(0,R)$, counting multiplicity.

Defining $G:\HBerk\to\RR$ exactly as in the proof of Theorem~\ref{expbound}, then, we have $$G(\zeta(0,R)) \geq m\log |N_0^-(\phi,\zeta(0,R),0) - N_0^-(\phi,\zeta(0,R),\infty)| =m\log m > |\phi'(0)| = G(\zeta(0,1)).$$ Thus, there is some radius $S\in [1,R)$ at which $G(\zeta(0,r))$ begins to increase.  We again conclude that inequality~\eqref{eq:Sbound} holds.  Because $\phi$ has no poles in $W=\Dbar(0,S)$, it follows that $$m[1 + N(\phi',W,0)] + (1-m)N(\phi,W,0)\geq 1.$$ Again setting $M\geq 0$ to be the number of critical points in $W$ that are not zeros of $\phi$, we use the fact that $N(\phi,W,0) > 1$ to obtain inequality~\eqref{eq:Mbound2}, which then implies $$m(M-1) + N(\phi,W,0)\geq 1.$$ On the other hand, since $W\subseteq V$, we have $N(\phi,W,0)\leq m$, and therefore $M \geq 1/m > 0$.  That is, there is a critical point in $W$ that is not a zero of $\phi$, and the conclusion follows.\end{proof}

\begin{remark}
\label{rem:polymore}
In the proof of Theorem~\ref{th:polys},
the only use of the hypothesis that $\phi$ is a polynomial
is to conclude that the connected affinoid $V$ is in fact a disk
$D(0,R)$, which in turn obviously implies that there are no
poles inside $D(0,R)$.
Thus, the sharper bound of Theorem~\ref{th:polys} actually holds
any time the region $V$ in the proof of Theorem~\ref{expbound}
is a disk, even when $\phi$ is a rational function.
\end{remark}

The bound in Theorem~\ref{th:polys} is sharp, as the following
example shows.

\begin{example}
Let $K$, $|\cdot|$, $p$, and $d$ be as in Theorem~\ref{th:main}.
Let $m$ be an integer minimizing $|m|^m$ for $1\leq m\leq d$.
We will construct
a PCF polynomial $\phi(z)\in K[z]$ with a fixed point at $z=0$
for which the multiplier satisfies $|\phi'(0)|=|m|^m$, showing the sharpness of Theorem~\ref{th:polys}. We leave the verification of these examples to the reader.

If $|m|\geq |d|$, then by minimality we must have either $|m| = |d| = 1$ or $m = d$, or else we have $|m|^m > |d|^d$. If $|d|=1$, then \[\phi(z) = (z+1)^d - 1\] has a fixed point at $0$ of multiplier $d$, but being conjugate to $z\mapsto z^d$, $\phi$ is also PCF. On the other hand, if $|d|<1$, then we may choose \[\phi(z)=\frac{d^d}{(1-d)^{d-1}}(z-1)^{d-1}z.\] 

In the case that $|m| < |d| \leq 1$, we
define
$$\phi(z) = a z (z-1)^{m-1} (z-b)^{d-m} \in K[z],$$
where we will choose $a,b\in K$ by
\[a = \alpha^{-1} (\alpha-1)^{-(m-1)} (\alpha-b)^{-(d-m)}\quad\text{and}\quad b = \frac{\alpha(d+1-m - d\alpha)}{1-m\alpha},\]
where $\alpha$ satisfies
\begin{multline}
\label{eq:alphapoly}
 d^d (d-m)^{d-m}\alpha^{d-m+2}(\alpha-1)^{d-1} (1-m\alpha)^{m-1}
\\ -
(1-m)^{m-1} (1-d\alpha)^{d-1} (d+1-m-d\alpha)^{d-m}=0.
\end{multline}
The reader can check that $\phi$ has no critical points other than $1$, $b$, $\infty$,
$\alpha$, and $\beta$, for 
\[\beta = \frac{d+1-m - d\alpha}{d(1-m\alpha)},\]
and from this that $\varphi$ is PCF.
\end{example}

Even if $\phi$ is not a polynomial,
the bound for $|\phi'(\gamma)|$ in
Theorem~\ref{expbound},
and hence the constant $\epsilon_{p,d}$
in Theorem~\ref{th:main}, can be improved to
\begin{equation} \label{epsilonprime}
\epsilon'_{p,d}=\min\big\{|m|^L |\ell|^{\ell-L} :
1\leq \ell \leq n \leq d \text{ and } 1\leq m \leq n \big\},
\end{equation}
where $L$ is defined to be $\lceil (n-1)/2 \rceil$,
with $\lceil x \rceil$ denoting the smallest integer greater
than or equal to $x$.
The proof of this bound is more complicated than the
proof of Theorem~\ref{th:main}, and the new constant $\epsilon'_{p,d}$
is not a substantial improvement over $\epsilon_{p,d}$ and does not appear to be quite sharp. For these reasons, we omit the derivation of this refinement.

Although the quantities $\epsilon_{p, d}$ in Theorem~\ref{th:main} (and the refinement mentioned in the previous paragraph) are probably not sharp, it is worth noting that optimal bounds for rational functions cannot possibly be as strong as the bounds for polynomials presented in Theorem~\ref{th:polys}. We illustrate this with the following example.

\begin{example}
\label{ex:ratl}
Let $K$, $|\cdot|$, $p$, and $d$ be as in Theorem~\ref{th:main}, and suppose that $m$ is an integer satisfying $1\leq m\leq d$ and
$|m|<|d|<1$.  Define
$$\phi(z) = \frac{a z (z-1)^{d-1}}{(z-b)^{d-m}} \in K(z)$$
for some $a,b\in K$ to be chosen shortly.
The only critical points of $\phi$ other than $1$, $b$, and $\infty$
are the two roots $\alpha,\beta\in K$ of
$$mz^2 -(db-d+m+1)z + b.$$
If we declare that $\alpha=\beta\neq 0$, it follows quickly
that $b=m\alpha^2$, and then that $\alpha$ satisfies
$$md \alpha^2 - 2m\alpha - d + m + 1 =0.$$
The Newton polygon of this equation
indicates that such an $\alpha\in K$ exists
with $|\alpha|=|md|^{-1/2}$.  Choosing
$b=m\alpha^2$ and $a=\alpha^{-1} (\alpha-1)^{1-d} (\alpha-b)^{d-m}$,
where the latter choice guarantees that $\phi(\alpha)=1$ and hence
that $\phi$ is PCF,
we have $|b|=|d|^{-1}$ and $|a|=|md|^{-m/2}$.
Thus, the multiplier of the fixed point at $0$ has absolute
value $|\phi'(0)| = |d|^{d-m/2} |m|^{m/2}$.

Although this multiplier does not attain the bound of
$\epsilon'_{p,d}$ in \eqref{epsilonprime}, it is in general strictly smaller than the
polynomial bound $\epsilon_{p,d}^{\mathrm{poly}}$ of Theorem \ref{th:polys}.
For example, if $K=\CC_p$ and $d=p^e + p^{e-1}$ for some
positive integer $e$, then
$\epsilon_{p,d}^{\mathrm{poly}} = |p|^{v}$, where
$v=p^{e-1}\max\{ep,(e-1)(p + 1)\}$.
Meanwhile, the choice of $m=p^e<d$ gives
$|d|^{d-m/2} |m|^{m/2} = |p|^w$, where
$w=p^{e-1} (2ep - p + 2e-2)/2$.
If $e>1 + p/2$, then $|p|^w < |p|^v$.

\end{example}

\section{Periodic Points}\label{perpts}

As noted in the introduction,
we would like to remove the dependence
on $n$ of the bound $\epsilon_{p,d^n}$
in order to prove Theorem~\ref{introper}.
The following lemma allows us to achieve this goal, at least
in the situation when $\epsilon_{p,d}=1$.

\begin{lemma}\label{lem:compositions}
Let $p$ and $K$ be as in Theorem~\ref{th:main}, and
suppose that $\phi(z)\in K(z)$ can be written as a composition of
rational functions, each of degree less than $p$.  Then
$|\deg_{\zeta,\vec{v}}(\phi)|=1$ for all $\zeta\in\PBerk$
and all directions $\vec{v}$ at $\zeta$.
\end{lemma}

\begin{proof}
Write $\phi=\psi_j\circ \cdots \circ\psi_1$, where
$d_i=\deg\psi_i\leq p-1$.  Given any $\zeta\in\PBerk$
and direction $\vec{v}$ at $\zeta$,
write $\xi_1=\zeta$ and
$\vec{w}_1 = \vec{v}$; then set $\xi_i = \psi_i(\xi_{i-1})$
and $\vec{w}_i = \psi_{i,*}(\vec{w}_{i-1})$
for all $i=2,\ldots,j$.  Then
$$\deg_{\zeta,\vec{v}} (\phi) = \prod_{i=1}^j \deg_{\xi_i,\vec{w}_i}(\psi_i)$$
by Lemma~\ref{lem:mult}.  However, each integer
$\deg_{\xi_i,\vec{w}_i}(\psi_i)$
lies between $1$ and $d_i\leq p-1$.  Hence,
\[ |\deg_{\zeta,\vec{v}}(\phi)| = \prod_{i=1}^j |\deg_{\xi_i,\vec{w}_i}(\psi_i)|
= 1^j = 1
\qedhere
\]
\end{proof}

\begin{proof}[Proof of Theorem~\ref{introper}]Let $\gamma$ be a periodic point of $\phi$ of period $n$ and multiplier $\lambda$, with $0<|\lambda| < 1$. Then $\gamma$ is an attracting fixed point for $\phi^n$. If $p = 0$, then $\epsilon_{p,d} = 1$ for all $d$, and by Theorem~\ref{th:main}, $\gamma$ strictly attracts a critical point of $\phi^n$ under the iteration of $\phi^n$. If $p > d$, then by Lemma~\ref{lem:compositions}, we have $|\deg_{\zeta,\vec{v}}(\phi^n)|=1$ for all $\zeta\in\PBerk$ and directions $\vec{v}$. Therefore, applying Theorem~\ref{expbound} to $\phi^n$, which is a rational function of degree $d^n$, $\gamma$ strictly attracts a critical point of $\phi^n$, under the iteration of $\phi^n$.

However, by the chain rule, critical points of $\phi^n$ are precisely points $x\in\PK$ for which at least one of $x, \phi(x), ..., \phi^{n-1}(x)$ is a critical point of $\phi$.  It follows that $\gamma$ strictly attracts a critical point of $\phi$ under the iteration of $\phi^n$, and hence that the periodic cycle of $\gamma$ strictly attracts a critical point of $\phi$ under the iteration of $\phi$.\end{proof}

Corollary~\ref{cor:attracting2d-2} follows immediately from Theorem~\ref{introper}, because under those hypotheses, there is a distinct critical point associated to each attracting cycle. The bound of $2d-2$ is simply the number of critical points of a rational function $\phi\in K(z)$ of degree $d$, counted with multiplicity. It is worth noting that, just as in the complex case, this result bounds the number of attracting \emph{cycles}, not the number of points in those cycles.

\begin{proof}[Proof of Corollary~\ref{cor:thurstonmodp}] Let $F$ be the prime field, and let $\phi$ be a non-isotrivial PCF map defined over a transcendental extension of $F$.  Note that it suffices to prove the statement in the case of transcendence rank 1. 
In general, the rational function $\phi$ corresponds to a positive-dimensional PCF subvariety $Y_\phi\subseteq \mathcal{M}_d$ defined over $\overline{F}$, and if any of the symmetric functions of the multipliers are non-constant on $Y_\phi$, then there exists a curve $C\subseteq Y_\phi$ upon which they are non-constant. The generic point on this curve is a $\overline{F}(C)$-rational point on $\mathcal{M}_d$, which corresponds to a PCF map simply because it satisfies the same critical orbit relations as $\phi$.

So now we assume without loss of generality that $K=\overline{F}(C)$ for some curve $C/\overline{F}$. Passing to a finite extension of $K$, for fixed $n$, we can assume that all of the multipliers of the $n$-periodic points are defined over $K$. If one of these, say $\lambda\in K$, is not contained in $\overline{F}$, then it vanishes at some point in $C$, and hence there is a valuation $v$ of $K$ such that $0<|\lambda|_v<1$. Applying Theorem~\ref{introper} we see that the cycle in question strictly attracts a critical point, contradicting the assumption that $\phi$ was PCF. So we have shown that the multipliers of the $n$-periodic orbits are all $\overline{F}$-rational and, since $n$ was arbitrary, we are done.
\end{proof}

\begin{remark}\label{finiteMc}Although the set of all multipliers of a PCF map in characteristic $0$ or $p > d$ only has to lie in an algebraic extension of the prime subfield, the \emph{symmetric functions} in the multipliers have to also lie in a field of definition of the map (see section 4.10 of~\cite{ads}). Since the map is always defined over a finite-type extension of the prime field, generated by its coefficients, the symmetric functions in the multipliers all lie in an extension of the prime subfield that is both algebraic and finite-type, that is a finite extension. This extension depends on the map but contains the symmetric functions in the multipliers of all periods.\end{remark}

If we had an analogue of McMullen's theorem in positive characteristic, then Corollary~\ref{cor:thurstonmodp} would imply that any PCF map in $K(x)$ of degree $d > p$ must have constant multipliers, and hence must be either Latt\`{e}s or isotrivial.  \emph{A priori}, however, there might be other exceptional varieties for the multiplier spectrum maps over fields of characteristic $p$. Fortunately, in one case we have a McMullen-type result independent of characteristic, giving us rigidity.

\begin{cor}\label{rigid2}Suppose that $K$ is a global function field of characteristic $0$ or $p \geq 3$.  Then any quadratic PCF map over $K$ is isotrivial, and thus defined over a finite extension of the prime field (i.e. a number field in characteristic $0$ and a finite field in characteristic $p \geq 3$) after an appropriate change of coordinates.\end{cor}

\begin{proof}By results of Silverman~\cite{silverman}, extending earlier work of Milnor~\cite{milnor}, the multiplier spectrum map $$\Lambda_1:\mathcal{M}_2\to\AA^3, \phi \mapsto (\lambda_1 + \lambda_2 + \lambda_3, \lambda_1\lambda_2 + \lambda_1\lambda_3 + \lambda_2\lambda_3, \lambda_1\lambda_2\lambda_3)$$ defines an isomorphism over $\ZZ$ between $\mathcal{M}_2$ and the plane in $\AA^3$ defined by the equation $\lambda_1\lambda_2\lambda_3 = \lambda_1 + \lambda_2 + \lambda_3 - 2$. Since the image of any PCF map under $\Lambda_1$ will be defined over an algebraic extension of the prime field, so is the corresponding point in $\mathcal{M}_2(K)$ by Corollary~\ref{cor:thurstonmodp}. Finally, since the field of moduli of the map is a finite extension of the prime field, so is the field of definition.\end{proof}

Just as Theorem~\ref{th:main} fails for function fields of characteristic $p$ with $0 < p \leq d$, so does Corollary~\ref{rigid2}. The polynomial $z^{p} + tz$, defined over the completed algebraic closure of $\FF_p((t))$, is PCF since the only critical point is $\infty$, but is non-isotrivial since the multiplier at $0$ is nonconstant.

\section{PCF maps and repelling cycles}\label{goodreduction}
It follows from Theorem~\ref{introper} that if $\phi(z)\in K(z)$ is a PCF map of degree $d\geq 2$, and if the residue characteristic of $K$ is 0 or $p>d$, then
$\phi$ has no periodic cycles that are attracting but not superattracting. Example~\ref{rlex} shows that such a map $\phi$ may have repelling cycles. However, the
following result shows that if $\phi$ is a polynomial, then it also has no repelling cycles.

\begin{theorem}\label{integrality}
Let $p$ and $K$ be as in Theorem~\ref{th:main},
let $\phi\in K[z]$ be a PCF polynomial
of degree $d\geq 2$,
and suppose either that $p=0$ or that $\phi$ can be written
as a composition of polynomials of degree less than $p$.
Then $\phi$ has potentially good
reduction, and in particular it has no repelling cycles.
\end{theorem}

\begin{proof}
Suppose $\phi$ does not have potentially good reduction.
After a change of coordinates, we may assume that $0$ is a fixed
point of $\phi$ and that $\phi$ is monic.  Hence, we may write
$\phi(z) = \sum_{i =   1}^{d}a_{i}z^{i}$, with $a_d=1$.
By our supposition, $\phi$ does not have good reduction in this
coordinate, and hence $|a_i|>1$ for some $1\leq i<d$.  As a result,
$\phi$ has a zero of some maximum absolute value $r>1$.

By Lemma~\ref{lem:compositions}, the tangent direction multiplicities
of $\phi$ at every Berkovich point are never divisible by $p$,
and hence for any $s>0$, the Weierstrass degree of $\phi$ on $D(0,s)$
is never divisible by $p$.  Thus, because this same Weierstrass
degree is the smallest integer $n\geq 1$ for which
$|a_n|s^n$ attains its maximum value, it follows that
$n$ is also the smallest integer for which
$|na_n|s^{n-1}$ attains its maximum value.  In other words,
the Newton polygon of $\phi'$ is exactly the Newton polygon
of $\phi$, but shifted to the left by one unit.  In particular,
$\phi$ and $\phi'$ have the same number of zeros (counting multiplicity)
of absolute value $r$.

For any zero $\alpha$ of $\phi$ with $|\alpha|=r$, consider the
polynomial $\psi(z)=\phi(z+\alpha)$.  Because $\psi(0)=0$ but
the tangent multiplicities of $\psi$ are prime to $p$, the same
argument as in the previous paragraph shows that the Newton polygon
of $\psi'$ is simply that of $\psi$ shifted one unit to the left,
and hence $\psi'$ has one fewer zero in $D(0,r)$ than $\psi$ does.
That is, $\phi$ has one fewer critical point in $D(\alpha,r)$ than
it has zeros.

Combining the conclusions of the previous two paragraphs, then,
it follows from the pigeonhole principle that there is some critical
point $\beta$ of $\phi$ such that $|\beta|=r$, but $\phi$ has no
zeros in $D(\beta,r)$.  It follows by induction on $m\geq 1$, then,
that
$$|\phi^m(\beta)| = \prod_{i=1}^d |\phi^{m-1}(\beta)-\alpha_i|
= r^{d^m},$$
where $\alpha_1,\ldots,\alpha_d$ are the zeros of $\phi$.
Since $r>1$, we see that the critical point
$\beta$ is strictly attracted to $\infty$, and hence $\phi$ is not PCF.
\end{proof}

\begin{remark}
Theorem~\ref{integrality} can be easily extended
to any PCF polynomial $\phi\in K[z]$
of degree $d\geq 2$ for which every Berkovich tangent multiplicity
$m=\deg_{\zeta,\vec{v}}(\phi)$ satisfies either $|m|=1$, $|m|>|d|$, or
$m=d$.  After all, under those weaker hypotheses, if the Newton polygon
of $\phi'$ is not simply a shift of the Newton polygon of $\phi$,
then $\phi$ has a critical point $\beta$ of absolute value
\emph{strictly} larger than the maximum absolute value $r$ of the roots
of $\phi$.  Thus, if $\phi$ had bad reduction, and hence $r>0$,
$\beta$ would again be strictly attracted to $\infty$.

In particular, if $K$ is a $p$-adic field, then
Theorem~\ref{integrality} applies to polynomials of $p$-power degree.
This observation reinterprets a result of Epstein~\cite{epstein2}, who
uses it to prove Thurston's rigidity using algebraic methods in the
case of polynomials of prime-power degree.
Unfortunately, one cannot combine the $p$-power result with
Theorem~\ref{integrality}
to prove an analogous theorem for, say, polynomials of degree $p^e m$,
where $m<p$.  Indeed, there are counterexamples to Epstein's statement in
cases where $p|d$ but $d$ is not a power of $p$.
\end{remark}

\begin{remark}
The proofs of Theorems~\ref{th:main} and~\ref{integrality} and Lemma~\ref{lem:compositions} show that the failure of an (insufficiently) attracting periodic point
to strictly attract a critical point, and the failure of bad reduction for a polynomial to force the point at infinity to strictly attract a critical point, can
only occur in the presence of wild ramification. Indeed, such pathological dynamics can only arise when a tangent direction $\vec{v}$ at a Berkovich point $\zeta$
has multiplicity $\deg_{\zeta,\vec{v}}(\phi)$ divisible by $p$. Without loss, we may assume that $\zeta$ is of type~II, and after (possibly different) coordinate
changes on the domain and range, we may assume that $\zeta=\phi(\zeta)=\zeta(0,1)$.  In those coordinates, having a tangent multiplicity divisible by $p$
corresponds to wild ramification of the reduced map $\overline{\phi}\in k(z)$ over $\PP^1(k)$, where $k$ denotes the residue field. In this light,
Lemma~\ref{lem:compositions} shows simply that a composition of maps of degree less than the residue characteristic can never be wildly ramified. Thus, the
important property of low-degree maps is not so much the fact that the degree is small, but rather that such maps are exhibit only tame ramification over the residue field in any choice of coordinates.
\end{remark}

\textbf{Acknowledgements}
The authors would like to thank ICERM for its Semester Program on
Complex and Arithmetic Dynamics, during which some of this work was
completed. The second and third authors would also like to thank BIRS
for hosting them (11RIT155), during which some first steps were taken
toward this project.
The first author gratefully acknowledges the support of NSF grants
DMS-0901494 and DMS-1201341.  The second author gratefully
acknowledges the support of NSERC of Canada during some of this
project. The third author gratefully acknowledges the support of NSF
grant DMS-0852826.

\end{document}